\documentclass[12pt]{article}
\usepackage{amsmath,amssymb,amsthm,graphics,color}
\title{Rotor Walks and Markov Chains}
\author{Alexander E. Holroyd and James Propp}
\date{28 April 2009 (revised 2 April 2010)}

\newtheorem{thm}{Theorem}
\newtheorem{lemma}[thm]{Lemma}
\newtheorem{prop}[thm]{Proposition}
\newtheorem{cor}[thm]{Corollary}

\newcommand{\prob}{\mathbb{P}}
\newcommand{\expe}{\mathbb{E}}
\renewcommand{\P}{\mathbb{P}}
\newcommand{\E}{\mathbb{E}}
\newcommand{\Z}{\mathbb{Z}}
\newcommand{\N}{\mathbb{N}}
\newcommand{\R}{\mathbb{R}}

\newcommand{\dof}{\bf\boldmath}
\newcommand{\ind}{\mathbf{1}}

\newcommand{\Khitprob}{K_1}
\newcommand{\Khittime}{K_2}
\newcommand{\Kstatvec}{K_3}
\newcommand{\Kstatprob}{K_4}
\newcommand{\Knumvistf}{K_5}
\newcommand{\Kstack}{K_6}

\newcommand{\numvis}{g}
\newcommand{\tmin}{s}
\newcommand{\tmax}{t}

\newcounter{mycount}
\newenvironment{mylist}{\begin{list}{(\roman{mycount})}
{\usecounter{mycount}\itemsep 0pt}}{\end{list}}

\begin{document}
\maketitle

\begin{center}
Dedicated to Oded Schramm, 1961--2008
\end{center}

\begin{abstract}
The rotor walk is a derandomized version of the random walk on
a graph.  On successive visits to any given vertex, the walker
is routed to each of the neighboring vertices in some fixed
cyclic order, rather than to a random sequence of neighbors.
The concept generalizes naturally to countable Markov chains.
Subject to general conditions, we prove that many natural
quantities associated with the rotor walk (including normalized
hitting frequencies, hitting times and occupation frequencies)
concentrate around their expected values for the random walk.
Furthermore, the concentration is stronger than that associated
with repeated runs of the random walk; the discrepancy is at
most $C/n$ after $n$ runs (for an explicit constant $C$),
rather than $c/\sqrt{n}$.
\end{abstract}

\renewcommand{\thefootnote}{}
\footnotetext{{\bf\hspace{-6mm}Key words:} rotor router; quasirandom;
Markov chain; discrepancy}
\footnotetext{{\bf\hspace{-6mm}2000 Mathematics Subject
Classifications:} 82C20; 20K01; 05C25}
\footnotetext{\hspace{-6mm}AEH funded in part by Microsoft Research
and an NSERC Grant.}
\footnotetext{\hspace{-6mm}JP funded in part by an NSF grant.}
\renewcommand{\thefootnote}{\arabic{footnote}}

\section{Introduction} \label{results}

Let $X_0,X_1,\ldots$ be a Markov chain on a countable set
$V$ with transition probabilities $p:V\times V\to[0,1]$
(see e.g.~\cite{norris} for background).
We call the elements of $V$ {\dof vertices}.
We write $\prob_u$ for the law of the Markov chain
started at vertex $u$ (so $\prob_u$-a.s.\ we have $X_0=u$).

The {\dof rotor-router walk} or {\dof rotor walk}
is a deterministic cellular automaton associated with the Markov chain,
defined as follows.
Assume that all transition probabilities $p(u,v)$ are rational
(later we will address relaxation of this assumption) and that for each
$u$ there are only finitely many $v$ such that $p(u,v)>0$.
To each vertex $u$ we associate a positive integer $d(u)$ and
a finite sequence of (not necessarily
distinct) vertices $u^{(1)},\ldots,u^{(d(u))}$,
called the {\dof successors} of $u$, in such a way that
\begin{equation}
\label{freq}
p(u,v)=\frac{\#\{i:u^{(i)}=v\}}{d(u)}\quad\text{for all }u,v\in V.
\end{equation}
(This is clearly possible under the given assumptions;
$d(u)$ may be taken to be the lowest common denominator of
the transition probabilities from $u$.)
The set $V$ together with the quantities $d(u)$
and the assignments of successors
will sometimes be called the {\dof rotor mechanism}.
\begin{figure}
\centering
\resizebox{8cm}{!}{\includegraphics{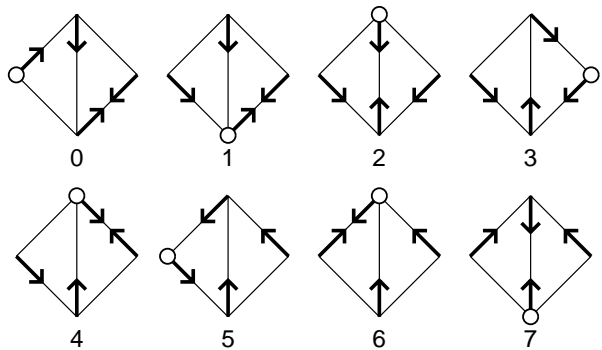}}
\caption{Steps $0,\ldots,7$ of a rotor walk associated with
the simple random walk on a graph with $4$ vertices.
The thin lines represent the edges of graph,
the circle is the particle location,
and the thick arrows are the rotors.
The rotor mechanism in this case is such that
each rotor successively points to the vertex's neighbors
in anticlockwise order.}\label{example}
\end{figure}

A {\dof rotor configuration} is a map $r$ that assigns to each vertex
$v$ an integer $r(v)\in\{1,\ldots, d(v)\}$.  (We think of an arrow or
{\dof rotor} located at each vertex, with the rotor at $v$ pointing to
vertex $v^{(r(v))}$).  We let a rotor configuration evolve in time, in
conjunction with the position of a particle moving from vertex to
vertex: the rotor at the current location $v$ of the particle is
incremented, and the particle then moves in the new rotor direction.
More formally, given a rotor mechanism, an initial particle
location $x_0\in V$ and an initial rotor configuration $r_0$, the
{\dof rotor walk} is a sequence of vertices $x_0,x_1,\ldots\in V$
(called {\dof particle locations}) together with rotor configurations
$r_0,r_1,\ldots$, constructed inductively as follows.  Given $x_t$ and
$r_t$ at time $t$ we set:
\begin{mylist}
\item
$r_{t+1}(v):=
\begin{cases}
(r_t(v)+1) \mod d(v), &v=x_t;\\
r_t(v), &v\neq x_t
\end{cases}$ \\
{\em (increment the rotor at the current particle
location)}; and
\item
$x_{t+1}:=(x_t)^{(r_{t+1}(x_t))}$ \\
{\em (move the particle in the new
rotor direction).}
\end{mylist}
See Figure \ref{example} for a simple illustration of the mechanism.

Given a rotor walk, write
$$n_t(v):=\#\big\{s\in[0,t-1]: x_s=v\big\}$$
for the number of times the particle visits vertex $v$ before
(but not including) time $t$.

We next state general results, Theorems
\ref{hit-prob}--\ref{stat-prob}, relating basic Markov chain
objects to their rotor walk analogues (under suitable
conditions).  We then state a more refined result (Theorem
\ref{hit-prob-log}) for the important special case of simple
random walk on $\Z^2$, followed by extensions to infinite times
(Theorem \ref{hit-prob-tf}) and irrational transition
probabilities (Theorem \ref{stack}).  We postpone discussion of
history and background to the end of the introduction, and
proofs to the later sections.

\subsection{Hitting probabilities}
\label{hitting-probabilities}

Let $T_v:=\min\{t\geq 0:X_t=v\}$ be the first hitting time of
vertex $v$ by the Markov chain (where $\min\emptyset:=\infty$).
Fix two distinct vertices $b,c$ and consider the hitting
probability
\begin{equation}\label{defofh}
h(v)=h_{b,c}(v):=\prob_v(T_b<T_c).
\end{equation}
Note that $h(b)=1$ and $h(c)=0$. In order to connect hitting
probabilities with rotor walks, fix a starting vertex
$a\not\in\{b,c\}$, and modify the transition probabilities from
$b$ and $c$ so that $p(b,a)=p(c,a)=1$. (Thus, after hitting $b$
or $c$, the particle is returned to $a$.) Note that this
modification does not change the function $h$. Modify the rotor
mechanism accordingly by setting $d(b)=d(c)=1$ and
$b^{(1)}=c^{(1)}=a$. Let $x_0,x_1,\ldots$ be a rotor walk
associated with the modified chain. The following is our most
basic result.

\begin{thm}[Hitting probabilities]
\label{hit-prob}
Under the above assumptions, suppose that the quantity
$$\Khitprob:=1+\frac{1}{2}\sum_{\substack{u\in V\setminus\{b,c\},\\ v\in V}}
d(u)p(u,v)|h(u)-h(v)|$$
is finite. Then for any rotor walk and all $t$,
$$\left| h(a)-\frac{n_t(b)}{n_t(b)+n_t(c)} \right|
\leq \frac{\Khitprob}{n_t(b)+n_t(c)}.$$
\end{thm}

Theorem \ref{hit-prob} implies that
the proportion of times that the rotor walk hits $b$ as opposed to $c$
converges to the Markov chain hitting probability $h(a)$,
provided the rotor walk hits $\{b,c\}$ infinitely often
(we will consider cases where this does not hold
in the later discussion on transfinite rotor walks).
Furthermore, after $n$ visits to $\{b,c\}$,
the discrepancy in this convergence is at most
$K/n$ for a fixed constant $K$.
In contrast, for the proportion of visits by the Markov chain itself,
the discrepancy is asymptotically a random multiple of $1/\sqrt n$
(by the central limit theorem).

The condition $\Khitprob<\infty$ holds in particular whenever
$V$ is finite, as well as in many cases when it is infinite;
for examples see~\cite{propp}.

In the case when the Markov chain (before modification) is a simple
random walk on an undirected graph $G=(V,E)$ (thus, $p(u,v)$ equals
$1/d(u)$ if $(u,v)$ is an edge, and $0$ otherwise, with $d(u)$ being the
degree of $u$), we obtain the particularly simple bound $\Khitprob
\leq 1+ \sum_{(u,v)\in E} |h(u)-h(v)|$.

Theorem \ref{hit-prob}
can be easily adapted to give similar results for
the probability of {\em returning\/} to $b$
before hitting $c$ when started at $a=b$,
and for the probability of hitting
one {\em set\/} of vertices before another.
This can be done either by adapting the proof
or by adding appropriate extra vertices
and then appealing to Theorem \ref{hit-prob}.
For brevity we omit such variations.

We next discuss extensions of Theorem \ref{hit-prob}
in the following directions:
hitting times and stationary distributions,
an example where $\Khitprob=\infty$,
cases where the particle can escape to infinity, and
irrational transition probabilities.

\pagebreak
\subsection{Hitting times} \label{hitting-times}

Fix a vertex $b$ and let
\begin{equation}\label{def-of-k}
k(v)=k_b(v):=\expe_v T_b
\end{equation}
be its expected hitting time.
Fix also an initial vertex $a\neq b$
and modify the transition probabilities from $b$ so that $p(b,a)=1$.
(Then $k(a)$ is also the expected return time from $b$ to $b$
in the reduced chain in which the vertices $a$ and $b$ are conflated.)
Let $x_0,x_1,\ldots$ be a rotor walk
associated with the modified chain.

\begin{thm}[Hitting times]
\label{hit-time} Under the above assumptions, suppose that $V$ is finite,
and let $$\Khittime:=
\max_{v\in V} k(v)+\frac{1}{2}\sum_{\substack{u\in V\setminus\{b\},\\v\in V}}
d(u)p(u,v)|k(u)-k(v)-1|.$$
Then for any rotor walk and all $t$,
$$\left|(k(a)+1)-\frac{t}{n_t(b)} \right|
\leq \frac{\Khittime}{n_t(b)}.$$
\end{thm}

Thus the average time for the rotor walk to get from $a$ to $b$
concentrates around the expected hitting time.
The ``$+1$" term corresponds to the time step to move from $b$ to $a$.

Note that, in contrast with Theorem \ref{hit-prob}, in the above result
we require $V$ to be finite.  Leaving aside some
degenerate cases, such a bound cannot hold when $V$ is infinite.
Indeed, if $V$ is infinite and the Markov chain is irreducible,
then $|(k(a)+1)n_t(b) - t|$ is unbounded in $t$,
since the rotor walk has arbitrarily long excursions
between successive visits to $b$; hence the conclusion of
Theorem \ref{hit-time} cannot hold (for any constant $K_2$) in this case.
In contrast, in the next result we again allow $V$ to be infinite.

\subsection{Stationary vectors} \label{stationary-vectors}

Suppose that the Markov chain is irreducible and recurrent,
and let $\pi:V\to(0,\infty)$ be a stationary vector
(so that $\pi p=\pi$ as a matrix product).
Let $x_0,x_1,\ldots$ be an associated rotor walk.
Fix two vertices $b\neq c$
and let $h=h_{b,c}$ be as in \eqref{defofh} above.
Also let $T^+_u:=\min\{t\geq 1: X_t=u\}$
denote the first return time to $u$,
and define the escape probability $e_{u,v}:=\P_u(T_v<T^+_u)$.

\begin{thm}[Occupation frequencies]
\label{stat-vec} \sloppy For any irreducible, recurrent Markov
chain, with the above notation, suppose that the quantity
$$\Kstatvec:=1+\frac12\Big(d(b)+d(c)+\sum_{u,v\in V}
d(u)p(u,v)|h(u)-h(v)|\Big)$$ is finite.
Then for all $t$,
$$\Big|\frac{n_t(b)}{\pi(b)}-\frac{n_t(c)}{\pi(c)}\Big|
\leq \frac{\Kstatvec}{\pi(b)e_{b,c}}.$$
\end{thm}
Thus the ratio of times spent at different vertices by the
rotor walk concentrates around the ratio of corresponding
components of the stationary vector.

Now suppose that the Markov chain is irreducible and positive recurrent,
and let $\pi$ be the stationary {\em distribution}
(so that $\sum_{v\in V} \pi(v)=1$).
Fix a vertex $b$ and let $k=k_b$ be as in \eqref{def-of-k}.
The following result
states that the proportion of time spent by the rotor walk at $b$
concentrates around $\pi(b)$.

\begin{thm}[Stationary distribution]
\label{stat-prob}
For an irreducible positive recurrent Markov chain
with $V$ finite, with the above notation, let
$$\Kstatprob:=\max_{v\in V} k(v)+\frac{1}{2}
\Big( \frac{d(b)}{\pi(b)} +
\sum_{u,v\in V} d(u)p(u,v)|k(u)-k(v)-1|\Big).$$
Then for all $t$,
$$\Big|\pi(b)-\frac{n_t(b)}{t}\Big| \leq \frac{\Kstatprob\pi(b)}{t}.$$
\end{thm}

\subsection{Logarithmic discrepancy for walks on $\Z^2$}  \label{walks-on-Z2}

While Theorem \ref{hit-prob} requires the quantity $\Khitprob$ to be finite,
experiments suggest that similar conclusions hold in many cases where it is
infinite. We next treat one interesting example in which such a conclusion
provably holds, but with an additional logarithmic factor in the bound on the
discrepancy.  (Additional such examples will appear in \cite{propp}.)

Consider simple symmetric random walk on the square lattice
$\Z^2$. That is, let $V=\Z^2$, and let $p(u,v):=1/4$ for all
$u,v\in V$ with $\|u-v\|_1=1$ and $p(u,v):=0$ otherwise. Let
each rotor rotate anticlockwise; that is for each $u\in V$, we
set $d(u):=4$ and
\begin{equation}\label{z2-mech}
u^{(i)}:=u+(\cos \tfrac{i\pi}2,\sin \tfrac{i\pi}2),\quad i=1,\ldots ,4.
\end{equation}
Consider the particular initial rotor configuration $r$ given by
\begin{equation}
r\big((x,y)\big):=\Big\lfloor \tfrac12+
\tfrac{2}{\pi} \mbox{arg}(x-\tfrac12,y-\tfrac12)\Big\rfloor \mod 4
\label{z2-r}
\end{equation}
(where $\arg(x,y)$ denotes the angle $\theta\in[0,2\pi)$
such that $(x,y)=r(\cos \theta,\sin\theta)$ with $r>0$).
See Figure \ref{swastika}.
\begin{figure}
\centering
\resizebox{5cm}{!}{\includegraphics{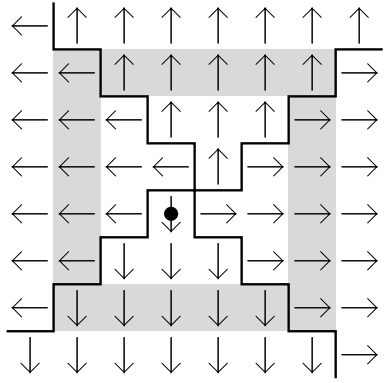}}
\caption{The initial rotor configuration in Theorem \ref{hit-prob-log}.
The dot shows the location of $(0,0)$.
The third layer is shaded (see the later proofs).}\label{swastika}
\end{figure}

Fix vertices $a,b,c$ of $\Z^2$ with $b\neq c$
and modify $p$ by setting $p(b,a)=p(c,a)=1$.
If $a=b$ then also split this vertex into two vertices $a$ and $b$,
let $b$ inherit all the incoming transition probabilities of the
original random walk, and let $a$ inherit the outgoing probabilities;
similarly if $a=c$.  Also modify the rotor mechanism and the rotor
configuration $r$ accordingly.

\begin{thm}[Hitting probabilities for walk on $\Z^2$]\label{hit-prob-log}
Let $a,b,c$ be vertices of $\Z^2$ with $b\neq c$,
and consider the rotor walk associated with the random walk,
rotor mechanism and initial rotor configuration described above,
started at vertex $a$.  Then for any $t$, with
$h(a) = h_{b,c}(a)$ and $n=n_t(b)+n_t(c)$,
$$\left| h(a)-\frac{n_t(b)}{n}
\right|\leq \frac{C \ln n}{n}.$$
Furthermore, $t\leq C' n^3$.
Here $C,C'$ are finite constants depending on $a,b,c$.
\end{thm}
In contrast to the above result for the rotor walk, for the
Markov chain itself, after $n$ visits to $\{b,c\}$ the
proportion of visits to $b$ differs from its limit $h(a)$ by
$K/\sqrt{n}$ in expected absolute value (by the central limit
theorem), while the median number of time steps needed to
achieve $n$ visits is at least $(K')^n$ where $K>0$ and $K'>1$
are constants depending on $a,b,c$. (The latter fact is an easy
consequence of the standard fact \cite{spitzer} that the
expected number of visits to the origin of $\Z^2$ after $t$
steps of random walk is $O(\ln t)$ as $t \rightarrow \infty$.)

Simulations suggest that a much tighter bound on the
discrepancy should actually hold in the situation of Theorem
\ref{hit-prob-log}, and in fact the results seem consistent
with a bound of the form $\mbox{const}/n$.  The rotor
configurations at large times are very interesting; see Figure
\ref{sim}. (Also compare with Figure \ref{sim-tf} below).
Further discussion of these issues will appear in \cite{propp}.
\begin{figure}
\centering
\resizebox{10cm}{!}{\includegraphics{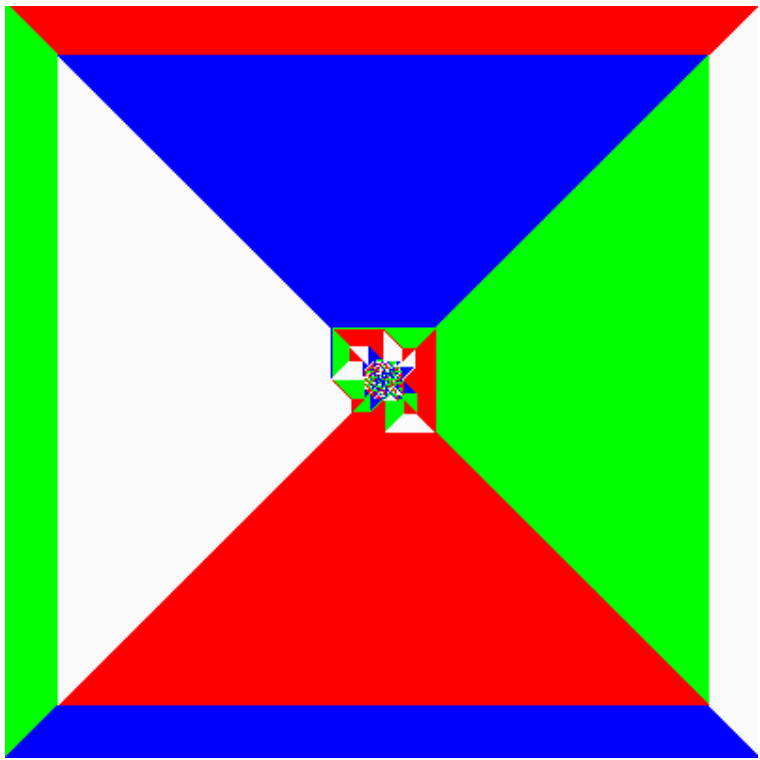}}
\caption{
The rotor configuration after 500 visits to $b$,
starting in the configuration of Figure \ref{swastika} with $a=c=(0,0)$ and $b=(1,1)$.
The rotor directions are: East=white, North=red, West=green, South=blue.
}\label{sim}
\end{figure}

\subsection{Transfinite walks} \label{transfinite-walks}

As mentioned above, Theorem \ref{hit-prob} implies
convergence of $n_t(b)/(n_t(b)+n_t(c))$
to $h(a)$ only if $n_t(b)+n_t(c)\to\infty$ as $t\to\infty$;
we now investigate when this holds,
and what can be done if it does not.
We say that a rotor walk is {\dof recurrent}
if it visits every vertex infinitely often,
and {\dof transient} if it visits every vertex only finitely often.

\begin{lemma}[Recurrence and transience]
\label{rec-trans} Any rotor walk associated with
an irreducible Markov chain is either recurrent or transient.
\end{lemma}
Note in particular that if $V$ is finite and $p$ is irreducible
then every rotor walk is recurrent.

Fix an initial rotor configuration $r_0$ and an initial vertex $x_0=a$.
Suppose that the rotor walk $x_0,x_1,\ldots$ is transient.
Then we can define a rotor configuration $r_\omega$ by
$r_\omega(v):=\lim_{t\to\infty} r_t(v)$
(the limit exists since the sequence $r_t(v)$ is eventually constant).
Now {\bf restart} the particle at $a$ by setting $x_\omega:=a$,
and define a rotor walk $x_\omega,x_{\omega+1},x_{\omega+2},\ldots$
according to the usual rules.
If this is again transient we can set
$r_{2\omega}:=\lim_{t\to\infty} r_{\omega+t}$
and restart at $x_{2\omega}:=a$ and so on.
Continue in this way up to the first $m$
for which the walk $x_{m\omega},x_{m\omega+1},\ldots$ is recurrent,
or indefinitely if it is transient for all $m$.
Call this sequence of walks a {\dof transfinite rotor walk}
started at $a$.

A {\dof transfinite time} is a quantity of the form $\tau=\omega^2$,
or $\tau=m\omega+t$ where $m,t$ are non-negative integers.
There is a natural order on transfinite times given by
$m\omega+t<m'\omega+t'$
if and only if either $m<m'$ or both $m=m'$ and $t<t'$,
while $m\omega+t<\omega^2$
for all $m$ and $t$.
For a transfinite walk and a transfinite time $\tau$ we write
$n_\tau(v)=\#\{\alpha<\tau: x_\alpha=v\}$
for the number of visits to $v$ before time $\tau$.
We sometimes say that the walk {\dof goes to infinity}
just before each of the times $\omega,2\omega,\ldots$
at which it is defined.

\begin{lemma}[Transfinite recurrence and transience]
\label{trans-rec-trans}
For an irreducible Markov chain and a transfinite rotor walk
started at $a$, for any transfinite time $\tau$,
either $n_\tau(v)$ is finite for all $v$
or $n_\tau(v)$ is infinite for all $v$.
Also there exists $M\in\{1,2,\ldots,\omega\}$ such that
$n_{M\omega}(v)$ is infinite for all $v$
and the rotor walk is defined at all $\tau<M\omega$.
\end{lemma}
Note that while it is not obvious how to use a finite computer running in a
finite time to compute transfinite rotor walks in general, it is at least
possible in certain settings, such as a random walk on the integers with a
periodic initial rotor configuration.

\begin{thm}[Transfinite walks and hitting probabilities]
\label{hit-prob-tf} Under the assumptions of Theorem \ref{hit-prob},
suppose further that $p$ is irreducible,
and that $$\limsup_{v\in V} h(v)=0.$$
Then for any transfinite time $\tau=m\omega+t$
at which all vertices have been visited only finitely often,
$$\left| h(a)-\frac{n_\tau(b)}{n_\tau(b)+n_\tau(c)+m}
\right|\leq \frac{\Khitprob}{n_\tau(b)+n_\tau(c)+m}.$$
\end{thm}
Thus the proportion of times the particle hits $b$
as opposed to hitting $c$ or going to infinity concentrates around $h(a)$.
Furthermore, Lemma \ref{trans-rec-trans} ensures
that $n_\tau(b)+n_\tau(c)+m\to\infty$ as $\tau\to M\omega$,
so that the proportion converges to $h(a)$.
The proof of Theorem \ref{hit-prob-tf} may be easily adapted to cover
the probability of hitting a single vertex $b$
as opposed to escaping to infinity.

Next, for a vertex $b$,
write $\numvis(v)=\numvis_b(v):=\E_v \sum_{t=0}^\infty \ind[X_t=b]$
for the expected total number of visits to $b$.
Note that this is finite for an irreducible, transient Markov chain.

\begin{thm}[Transfinite walks and number of visits]
\label{num-vis-tf} Consider an irreducible, transient Markov chain
and fix vertices $a,b$.  Suppose that $$\limsup_{v\in V} \numvis(v)=0.$$
Suppose moreover that the quantity
$$\Knumvistf := \sup_{v\in V} \numvis(v) + \frac12 \Big( d(b) +
\sum_{u,v\in V} d(u)p(u,v)\big|\numvis(u)-\numvis(v)| \Big)$$ is finite.
Then for any transfinite walk started at $a$,
and for any transfinite time $\tau=m\omega+t$ at which
all vertices have been visited only finitely often,
$$\left| \numvis(a)-\frac{n_\tau(b)}{m} \right| \leq \frac{\Knumvistf}{m}.$$
\end{thm}

It is natural to ask how recurrence and transience of rotor walks
are related to recurrence and transience
of the associated Markov chain.
The following variant of an unpublished result of Oded Schramm
provides an answer in one direction:
in a certain asymptotic sense,
the rotor walk is no more transient than the Markov chain.
For a transfinite rotor walk started at vertex $a$,
let $I_n$ be the number of times the walk goes to infinity
before the $n$th return to $a$ (i.e.\ $I_n:=\max\{m\geq 0:n_{m\omega}(a)<n\}$
--- this is well defined by Lemma \ref{trans-rec-trans};
recall that the walk is restarted at $a$ after each escape to infinity).
\begin{thm}[Transience density; Oded Schramm]\label{density}
Consider an irreducible Markov chain,
and an associated transfinite rotor walk started at vertex $a$.
With $I_n$ as defined above we have
$$
\limsup_{n\to\infty} \frac{I_n}{n} \leq \P_a(T^+_a=\infty).
$$
\end{thm}
In particular we note that for a recurrent Markov chain the right side
in Theorem \ref{density} is zero, so the sequence of escapes to
infinity has density zero in the sequence of returns to $a$.  On the
other hand, for a recurrent Markov chain it is possible for a rotor
walk to go to infinity, for example in the case of simple symmetric
random walk on $\Z$, with all rotors initially pointing in the same
direction.

Moreover, for simple random walk on $\Z^2$ with all rotors
initially in the same direction, the rotor walk goes to
infinity infinitely many times.  (To check this, suppose the
rotors rotate anticlockwise and initially point East.  Whenever
the particle's horizontal coordinate achieves a new maximum, it
is immediately sent directly Northwards to infinity.  This
happens infinitely often by Lemma \ref{trans-rec-trans}.)  See
Figure \ref{sim-tf} for a simulation of this remarkable
process, and see \cite{propp} for further discussion.

On the other hand, it should be noted that the rotor walk on
$\Z^2$ is recurrent for the initial configuration in Theorem
\ref{hit-prob-log} (see Figure \ref{swastika}). It is also
possible for the rotor walk to be recurrent for a transient
Markov chain, for example in the case of simple random walk on
an infinite binary tree, with all rotors arranged so as to next
send the particle towards the root. Landau and
Levine~\cite{landau-levine} studied the rotor walk on regular
trees in great detail, in particular identifying exactly which
sequences $(I_n)_{n\geq 0}$ are possible on the binary tree.
Further work on rotor walks on trees will appear in
\cite{angel-holroyd}.
\begin{figure}[t]
\centering
\resizebox{10cm}{!}{\includegraphics{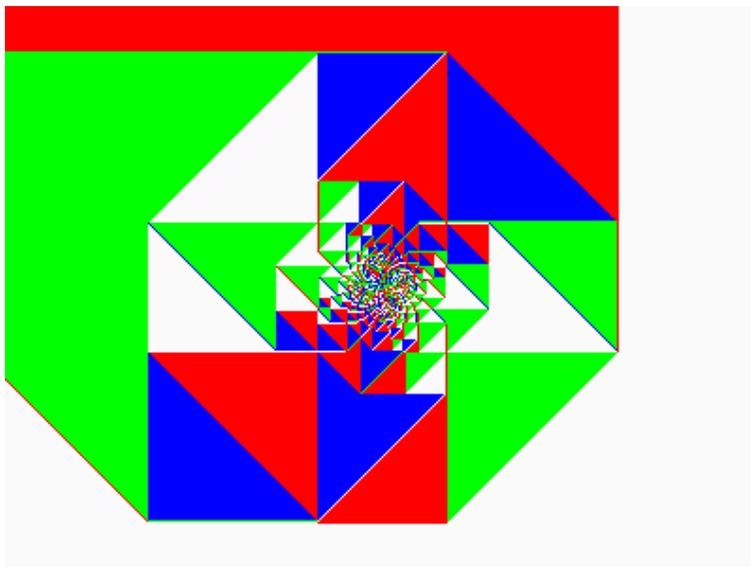}}
\caption{
The rotor configuration after $500$ restarts from $a=(0,0)$,
for the transfinite rotor walk on $\Z^2$ with all rotors
initially pointing East.
The rotor directions are: East=white, North=red, West=green, South=blue.
The red region extends infinitely far to the North.
}\label{sim-tf}
\end{figure}

\subsection{Stack walks} \label{stack-walks}

To generalize rotor walks to Markov chains with
irrational transition probabilities,
we must allow the particle to be routed to
a non-periodic sequence of vertices
on its successive visits to a given vertex.

Given a set $V$, a {\dof stack mechanism} is an assignment
of an infinite sequence of successors $u^{(1)},u^{(2)},\ldots$
to each vertex $u\in V$.
The {\dof stack walk} started at $x_0$ is
a sequence of vertices $x_0,x_1,\ldots$ defined inductively by
$$x_{t+1}:=x_t^{(n_t(x_t)+1)}$$
where
$$n_t(v):=\#\{s\in[0,t-1]:x_s=v\}.$$
(Note that, in the case of rational transition probabilities
considered previously, the rotor walk can be regarded as a special
case of a stack walk, with the periodic stacks given by $u^{(k
d(u)+j)}=u^{(j)}$ for $1\leq j\leq d(u)$ and $k\geq 0$.)

We illustrate the use of stacks with
Theorem \ref{stack} below on hitting probabilities.
The following will enable us to choose a suitable stack mechanism.
%
\begin{prop}[Low-discrepancy sequence]\label{seq}
Let $p_1,\ldots,p_n\in(0,1]$ satisfy $\sum_i p_i=1$.  There exists a sequence
$z_1,z_2,\ldots \in\{1,\ldots,n\}$ such that for all $i$ and $t$,
\begin{equation}\label{ineq}
 \Big|p_i t - \#\{s\leq t: z_s=i\}\Big|\leq 1.
\end{equation}
\end{prop}
Let $p$ be a Markov transition kernel on $V$,
and suppose that for each vertex $u$
there are only finitely many vertices $v$ such that $p(u,v)>0$.
We may then choose a stack mechanism according to Proposition \ref{seq}.
More precisely, for each vertex $u$,
enumerate the vertices $v$ such that $p(u,v)>0$ as $v_1,\ldots,v_n$,
and set $p_i=p(u,v_i)$.
Then let $u^{(j)}:=v_{z_j}$ where $z$ is the sequence
given by Proposition \ref{seq}.
Now let $a,b,c$ be distinct vertices
and assume that $p(b,a)=p(c,a)=1$ and $b^{(i)}=c^{(i)}=a$ for all $i$.
Write $h=h_{b,c}$.
\begin{thm}[Stack walks]\label{stack}
Under the above assumptions, suppose that
$$\Kstack:=1+\sum_{\substack{u\in V\setminus\{b,c\}, v\in V: \\ p(u,v)>0}}
|h(u)-h(v)|$$
is finite.  For the stack mechanism described above, and any $t$,
$$\left| h(a)-\frac{n_t(b)}{n_t(b)+n_t(c)}\right| \leq
\frac{\Kstack}{n_t(b)+n_t(c)}.
$$
\end{thm}

Proposition \ref{seq} can in fact be extended to the case of
infinite probability vectors~\cite{ahmp}, and
Theorem~\ref{stack} carries over straightforwardly to this
case. However, the result appears to have few applications in
this broader context, since $\sum_v |h(u)-h(v)|$ is typically
infinite when $u$ has infinitely many successors.

\subsection{Further Remarks.} \label{further-remarks}

\paragraph{History.} The rotor-router model was
introduced by Priezzhev, Dhar, Dhar and
Krishnamurthy~\cite{pddk} (under the name ``Eulerian walkers
model'') in connection with self-organized criticality. A
special case was rediscovered in~\cite{dtw} in the analysis of
some combinatorial games.  The present article reports the
first work on the close connection between rotor walks and
Markov chains, originating in discussions between the two
authors at a meeting in 2003. (Such a connection was however
anticipated in the ``whirling tours'' theorem of~\cite{dtw},
which shows that for random walk on a tree, the expected
hitting time from one vertex to another can be computed by
means of a special case of rotor walk; see also \cite{propp}.)
A special case of results presented here was reported
in~\cite{kleber}, and earlier drafts of the current work
provided partial inspiration for some of the recent progress
in~\cite{cooper-doerr-friedrich-spencer,cooper-doerr-spencer-tardos-0,
cooper-doerr-spencer-tardos,cooper-spencer,doerr-friedrich,hlmppw,
levine-peres-3,levine-peres-2,levine-peres}, which we discuss
below.

The idea of stack walks has its roots in Wilson's approach to
random walks via random stacks; see~\cite{wilson}.

\paragraph{Time-dependent bounds.}
We have chosen to focus on upper bounds of the form $K_i/n$,
where $K_i$ is a fixed constant not depending on time $t$. If
this latter requirement is relaxed, our proofs may be adapted
to give bounds that are stronger in some specific cases (at the
expense of less clean formulations). Specifically, in each of
Theorems \ref{hit-prob}--\ref{stat-prob} and \ref{stack}, the
claimed bound still holds if the relevant constant $K_i$ is
replaced with a modified quantity $K_i(t)$ obtained from $K_i$
by:
\begin{mylist}
\item multiplying the initial additive term ``1'' or ``$\max
    k(v)$'' or ``$\sup g(v)$'' by the indicator $\ind[x_t\neq x_0]$
(so that the term vanishes when the particle returns to its
starting point); and
\item multiplying the summand in the sum $\sum_{u,v}$ by
    $\ind[r_t(u)\neq r_0(u)]$ (so in particular terms
    corresponding to vertices $u$ that have not been visited by time $t$ vanish).
\end{mylist}
The same holds for Theorems \ref{hit-prob-tf} and
\ref{num-vis-tf} in the transfinite case, but replacing $t$
with $\tau$.

The above claims follow by straightforward modifications to our
proofs.  Indeed, our proof of Theorem \ref{hit-prob-log}
employs a special case of this argument.  These and other
refinements will be discussed more fully in the forthcoming
article \cite{propp}.

\paragraph{Abelian property.} \label{abelian-property}
The rotor-router model has a number of interesting properties
that will not be used directly in most of our proofs but which
are nonetheless relevant. In particular, it enjoys an ``Abelian
property'' which allows rotor walks to be parallelized.
Specifically, consider a Markov chain on a finite set $V$ with
one or more {\dof sinks}, i.e.\ vertices $s$ with $p(s,s)=1$,
and suppose that from every vertex, some sink is accessible (so
that the Markov chain eventually enters a sink almost surely).
Then we may run several rotor walks simultaneously as follows.
Start with an initial rotor configuration, and some
non-negative number of particles at each vertex.  At each step,
choose a particle and route it according to the usual rotor
mechanism; i.e.\ increment the rotor at its current vertex and
move the particle in the new rotor direction. Continue until
all particles are at sinks.  It turns out that the resulting
configuration of particles and rotors is independent of the
order in which we chose to route the particles.  This is the
Abelian property; see e.g.~\cite[Lemma 3.9]{hlmppw} for a proof
(and generalizations).

In the situation of Theorem \ref{hit-prob}, for example,
assume that $V$ is finite and the Markov chain is irreducible,
and then modify it to make vertices $b$ and $c$ sinks.
Start $n$ particles at vertex $a$ and perform simultaneous rotor walks.
The Abelian property implies that the number of particles eventually at $b$
is the same as the number $n_t(b)$ of times that $b$ is visited
when $n_t(b)+n_t(c)=n$ in the original set-up of Theorem \ref{hit-prob},
and the bound of Theorem \ref{hit-prob} therefore applies.

\sloppy A similar Abelian property holds
for the ``chip-firing'' model introduced by Engel \cite{engel-1,engel-2}
(later re-invented by Dhar \cite{dhar}
under the name ``abelian sandpile model''
as another model for self-organized criticality).
The two models have other close connections,
and in particular there is a natural group action
involving sandpile configurations acting on rotor configurations.
More details may be found in~\cite{hlmppw} and references therein.
Engel's work was motivated by an analogy
between Markov chains and chip-firing
(indeed, he viewed chip-firing as an ``abacus''
for Markov chain calculations).

\paragraph{Periodicity.} \label{periodicity} In the case when $V$ is
finite, we note the following very simple argument which gives
bounds similar to Theorems \ref{hit-prob}--\ref{stat-prob} but
with (typically) much worse constants. Since there are only
finitely many rotor configurations, the sequence of vertices
$((x_t,r_t))_{t\geq 0}$ is eventually periodic (with explicit
upper bounds on the period and the time taken to become
periodic which are exponentially large in the number of
vertices). Therefore the proportion of time $n_t(v)/t$ spent at
vertex $v$ converges as $t\to\infty$ to some quantity $\mu(v)$,
say, with a discrepancy bounded by $\mbox{const}/t$.
Furthermore, as a consequence of the rotor mechanism, we have
$\mu(u)=\sum_{v\in V} p(u,v) \mu(v)$ for all vertices $u$
(because after many visits to $u$, the particle will have been
routed to each successor approximately equal numbers of times).
Thus $\mu$ is a stationary distribution for the Markov chain.
This implies the bound in Theorem \ref{stat-prob}, except with
a different (and typically much larger) constant in place of
$\Kstatprob \pi(b)$. Similar arguments yield analogues of
Theorems \ref{hit-prob}--\ref{stat-vec}, but only in the case
where $V$ is finite.

\paragraph{Related work.} \label{related-work}
As remarked earlier, rotor walks on trees were studied in
detail by Landau and Levine~\cite{landau-levine}.  Further
results on rotor walks on trees will appear in a forthcoming
work of Angel and Holroyd~\cite{angel-holroyd}, and further
refinements and discussions of the results presented here will
appear in Propp~\cite{propp}.

Cooper and Spencer~\cite{cooper-spencer} studied the following
closely related problem. For the rotor walk associated with
simple symmetric random walk on $\Z^d$, start with $n$
particles at the origin, or more generally distributed in any
fashion on vertices $(i_1,\dots,i_d)$ with $i_1+\dots+i_d$
even, and apply one step of the rotor walk to each particle;
repeat this $t$ times. (It should be noted that the Abelian
property does {\em not} apply here --- the result is not the
same as applying $t$ rotor steps to each particle in an
arbitrary order; see~\cite{hlmppw}.) It is proved
in~\cite{cooper-spencer} (see Figure~8) that the number of
particles at a given vertex differs from the expected number of
particles for $n$ random walks by at most a constant (depending
only on $d$). Further more precise estimates are proved in
dimension $d=1$ in~\cite{cooper-doerr-spencer-tardos} and in
dimension $d=2$ in~\cite{doerr-friedrich}.

The following rotor-based model for internal diffusion-limited aggregation
(IDLA) was proposed by the second author, James Propp,
and studied by Levine and Peres
in~\cite{levine,levine-peres-3,levine-peres-2,levine-peres}.
Starting with a rotor configuration on $\Z^d$,
perform a sequence of rotor walks starting at the origin,
stopping each walk as soon as it reaches a vertex
not occupied by a previously stopped particle.
It is proved in~\cite{levine-peres} that,
as the number of particles $n$ increases,
the shape of the set of occupied vertices converges to
a $d$-dimensional Euclidean ball;
generalizations and more accurate bounds are proved
in~\cite{levine-peres-3,levine-peres-2}.

\section{Proofs of basic results} \label{proofs-of-basic-results}

Theorems \ref{hit-prob}--\ref{stat-prob} will all follow
as special cases of Proposition \ref{key} below,
and the remaining results will also follow by adapting the same proof.
For any Markov transition kernel $p$ and any function $f:V\to\R$
we define the {\dof Laplacian} $\Delta f:V\to\R$ by
\begin{equation}\label{def-of-lap}
\Delta f(u):=\sum_{v\in V} p(u,v)f(v)-f(u).
\end{equation}
\begin{prop}[Key bound]\label{key}
For any rotor walk $x_0,x_1,\ldots$ associated with $p$,
any function $f$ and any $t$ we have
$$\Big|\sum_{s=0}^{t-1} \Delta f(x_s)\Big|\leq |f(x_t)-f(x_0)|+
\frac12\sum_{u,v\in V}d(u)p(u,v)\big|f(u)-f(v)+\Delta f(u)\big|.$$
\end{prop}

The proofs of Theorems \ref{hit-prob}--\ref{stat-prob}
will proceed by applying Proposition \ref{key} to a suitable $f$.
The proof of Proposition \ref{key} will use the following simple fact.
\begin{lemma}\label{cyclic-sum}
  If $\sum_{i=1}^n a_i=0$ then
$\big|\sum_{i=1}^j a_i-\sum_{i=1}^k a_i\big|\leq \frac12\sum_{i=1}^n |a_i|$
for all $j,k\in[1,n]$.
\end{lemma}

\begin{proof}
We prove the stronger statement that
$|\sum_{i\in S} a_i|\leq \frac{1}{2}\sum_{i=1}^n |a_i|$
for any subset $S$ of $\{1,\ldots,n\}$:
assuming without loss of generality that $\sum_{i\in S}a_i$ is positive,
it is at most
$\sum_{a_i:a_i>0}a_i=\frac{1}{2}(\sum_{i:a_i>0} a_i-\sum_{i:a_i<0} a_i)
=\frac{1}{2}\sum_{i=1}^n |a_i|.$
\end{proof}

\begin{proof}[Proof of Proposition \ref{key}]
Recall that $r_0$ denotes the initial rotor configuration.
For a vertex $x$ and a rotor configuration $r$, consider the quantity
$$\Phi(x,r):=f(x)+\sum_{u\in V} \big[\phi(u,r(u))-\phi(u,r_0(u))\big]$$
where
$$\phi(u,j):=\sum_{i=1}^{j} \big[f(u)-f(u^{(i)})+\Delta f(u)\big].$$
Note that $\Phi(x,r_t)$ is finite
if $r_t$ is any rotor configuration encountered by the rotor walk,
since the only non-zero terms in the sum over $u$ are those
corresponding to vertices that the walk has visited
(this is the reason for including
the term ``$-\phi(u,r_0(u))$" in the above definition).
Note also that the definition of the Laplacian \eqref{def-of-lap}
and the rotor property \eqref{freq} imply for all $u\in V$ that
\begin{equation}\label{cyclic-phi}
  \phi(u,d(u))=0.
\end{equation}

Let us compute the change in $\Phi$ produced by
a step of the rotor walk from $(x_t,r_t)$ to $(x_{t+1},r_{t+1})$.
The only term in the sum over $u$ that changes
is the one corresponding to $u=x_t$, and thus
\begin{align*}
  \Phi(x_{t+1},r_{t+1})-\Phi(x_t,r_t)
  &=f(x_{t+1})-f(x_t)+[\phi(x_t,r_{t+1}(x_t)) - \phi(x_t,r_t(x_t))]\\
  &=f(x_{t+1})-f(x_t)+[f(x_t)-f(x_t^{(r_{t+1}(x_t))})+\Delta f(x_t)]\\
  &=\Delta f(x_t),
\end{align*}
where we have used \eqref{cyclic-phi} in the case when $r_{t+1}(x_t)=1$.
Therefore
$\Phi(x_t,r_t)-\Phi(x_0,r_0)=\sum_{s=0}^{t-1} \Delta f(x_s)$.
Also $\Phi(x_0,r_0)=f(x_0)$, so we obtain
\begin{equation}\label{almost-done}
  \sum_{s=0}^{t-1} \Delta f(x_s)
 =f(x_t)-f(x_0)+\sum_{u\in V} \big[\phi(u,r_t(u))-\phi(u,r_0(u))\big].
\end{equation}

In order to bound the last sum in \eqref{almost-done},
we use \eqref{cyclic-phi} together with Lemma \ref{cyclic-sum}
and the definition of $\phi$ to deduce
\begin{align}
\big|\phi(u,r_t(u))-\phi(u,r_0(u))\big|
&\leq \frac12\sum_{i=1}^{d(u)}\big|f(u)-f(u^{(i)})+\Delta f(u)\big| \nonumber\\
&=\frac12\sum_{v\in V}d(u)p(u,v)\big|f(u)-f(v)+\Delta f(u)\big|
\label{term-bound}
\end{align}
(since $d(u)p(u,v)$ is
the number of $i$ such that $u^{(i)}=v$).
We conclude by applying the triangle inequality to \eqref{almost-done}.
\end{proof}

\begin{proof}[Proof of Theorem \ref{hit-prob}]
  We will apply Proposition \ref{key} with $f=h_{b,c}$.
Note that $h(b)=1$ and $h(c)=0$,
while conditioning on the first step of the Markov chain
gives $h(u)=\sum_{v\in V} p(u,v)h(v)$ for $u\neq b,c$.
Hence, using $p(b,a)=p(c,a)=1$,
  $$\Delta h(u)=\begin{cases}
  0,&u\neq b,c;\\
  h(a)-1,&u=b;\\
  h(a),&u=c,
  \end{cases}$$
  and thus $\sum_{s=0}^{t-1} \Delta h(x_s) = h(a)[n_t(b)+n_t(c)]-n_t(b)$.

Turning to the other terms in Proposition \ref{key}, note that
$|h(x_t)-h(x_0)|\leq 1$, and $h(u)-h(v)+\Delta h(u)=0$ when
$u\in\{b,c\}$ and $v=a$. Substituting into Proposition
\ref{key} gives $\big| h(a)[n_t(b)+n_t(c)]-n_t(b) \big|\leq
\Khitprob$ as required.
\end{proof}

\begin{proof}[Proof of Theorem \ref{hit-time}]
  We will apply Proposition \ref{key} with $f=k_b$.  In this case
  $$\Delta k(u)=\begin{cases}
  -1,&u\neq b;\\
  k(a)&u=b,
  \end{cases}$$
  and thus $\sum_{s=0}^{t-1} \Delta k(x_s)
  =(n_t(b))(k(a)) + (t-n_t(b))(-1)
  =(k(a)+1)n_t(b)-t$.
  Substituting into Proposition~\ref{key}
  and using $|k(x_t)-k(x_0)|\leq \max_{v\in V} k(v)$
  and $k(b)-k(a)+\Delta k(b)=0$ completes the proof.
\end{proof}

To prove Theorem \ref{stat-vec}
we note some elementary facts about Markov chains.
\begin{lemma}\label{escape}
Let $b,c$ be two distinct vertices of an irreducible recurrent Markov chain,
and let $\pi$ be a stationary vector.  Then
$\pi(b)e_{b,c}= \pi(c)e_{c,b}$.
Also the hitting probabilities $h=h_{b,c}$ satisfy
$\Delta h(b)=-e_{b,c}$ and $\Delta h(c) = e_{c,b}$.
\end{lemma}

\begin{proof}
Let $N$ denote the number of visits to $c$
before the first return to $b$ when started from $b$.
It is a standard fact (see e.g.~Theorem 1.7.6 in~\cite{norris})
that $\E N=\pi(c)/\pi(b)$.
On the other hand $\P(N=n)=e_{b,c}(1-e_{c,b})^{n-1}e_{c,b}$ for $n\geq 1$,
so $\E N=e_{b,c}/e_{c,b}$, and the first claim follows.
For the remaining claims we compute $\Delta h$
by conditioning on the first step:
$\Delta h(b)=(1-e_{b,c})-h(b)=-e_{b,c}$
and $\Delta h(c) = e_{c,b}-h(c)=e_{c,b}$.
\end{proof}

\begin{proof}[Proof of Theorem \ref{stat-vec}]
  We will again apply Proposition \ref{key} with $f=h=h_{b,c}$
 (now without the restriction $p(b,a)=p(c,a)=1$).  Lemma \ref{escape} gives
  $$\Delta h(u)=\begin{cases}
  0,&u\neq b,c;\\
  -e_{b,c},&u=b;\\
  e_{c,b},&u=c,
  \end{cases}$$
and so $\sum_{s=0}^{t-1} \Delta h(x_s) = -n_t(b)e_{b,c}+n_t(c)e_{c,b}$.
In order to bound the terms in the last sum in Proposition \ref{key}
in the cases $u=b,c$
note that $|\Delta h(u)|\leq 1$ in these cases, and so, for $u=b,c$,
$$\sum_{v\in V} p(u,v)|h(u)-h(v)+\Delta h(u)|
\leq 1+\sum_{v\in V} p(u,v)|h(u)-h(v)|.$$
Hence Proposition \ref{key} gives
$$\big| n_t(b)e_{b,c}-n_t(c)e_{c,b}\big| \leq \Kstatvec.$$
Now divide through by $\pi(b)e_{b,c}$
(which equals $\pi(c)e_{c,b}$ by Lemma \ref{escape}).
\end{proof}

\begin{proof}[Proof of Theorem \ref{stat-prob}]
We will apply Proposition \ref{key} with $f=k=k_b$.
Note that $k(b)=0$, while $\E_b T^+_b=1+\sum_{v\in V} p(b,v)k(v)$.
Also we have $\E_b T^+_b=1/\pi(b)$ (see~\cite{norris}),
hence
 $$\Delta k(u)=\begin{cases}
  -1,&u\neq b;\\
  1/\pi(b)-1&u=b.
  \end{cases}$$
We bound the term for $u=b$ in Proposition \ref{key} thus:
$\sum_{v\in V} p(b,v)|k(b)-k(v)+\Delta k(b)|
\leq 1/\pi(b)+ \sum_{v\in V} p(b,v)|k(b)-k(v)-1|$.
We obtain
$$\Big|\frac{n_t(b)}{\pi(b)}-t\Big|\leq \Kstatprob,$$
and multiply by $\pi(b)/t$ to conclude.
\end{proof}

\section{Proofs for walks on $\Z^2$} \label{proofs-for-walks-on-Z2}

Our proof of Theorem \ref{hit-prob-log} is based on the two lemmas below.
For $k\geq 1$ we define the $k$th {\dof box} $B(k):=(-k,k]^2\cap\Z^2$
and the $k$th {\dof layer} $\partial B(k):=B(k)\setminus B(k-1)$.
See Figure \ref{swastika}.

\begin{lemma}\label{log-bound}
  Fix two distinct vertices $b,c$ of $\Z^2$,
and let $h=h_{b,c}$ be the hitting probability
for the simple random walk on the square lattice.
There exists $C=C(b,c)\in(0,\infty)$ such that for all positive integers $k$,
  $$\sum_{\substack{u,v\in B(k):\\ \|u-v\|_1=1}} |h(u)-h(v)| \leq C\ln k.$$
\end{lemma}

\begin{proof}
  Fix $b,c$ and write $C_1,C_2,\ldots$ for constants depending on $b,c$.
  We claim first that for all $v\in\Z^2$,
  \begin{equation}\label{h-formula}
     h(v)=C_1+e_{b,c}[a(v-c)-a(v-b)],
  \end{equation}
  where $a:\Z^2\to\R$ is the {\em potential kernel} of $\Z^2$.
(The function $a$ may be expressed as
$a(v):=\lim_{n\to\infty} \sum_{t=0}^n [\P(X_t=0)-\P(X_t=v)]$,
where $(X_t)$ is the simple random walk on $\Z^2$ ---
for more information see
e.g.~\cite[Ch.\ 3]{spitzer} or~\cite[Sect.\ 1.6]{lawler}.)
To check \eqref{h-formula}, we note the following facts about $a$.  Firstly,
  \begin{equation}\label{asym}
    a(v)=A+\tfrac{2}{\pi}
    \ln |v|+O(|v|^{-2})\quad \text{as }|v|\to\infty,
  \end{equation}
  where $|v|:=\|v\|_2$ and $A$ is an absolute constant
  (see~\cite[p.~39]{lawler}).
  Since $\frac{d}{dx} (A+ \tfrac{2}{\pi}\ln x) = \tfrac{2}{\pi}x^{-1}$ we deduce
  \begin{equation*}
    |a(v-c)-a(v-b)|\leq C_2|v|^{-1}.
  \end{equation*}
  Secondly, writing $\Delta$ for the Laplacian of the random walk on $\Z^2$,
  i.e.\ $\Delta f(u):=\tfrac14\sum_{v:\|u-v\|_1=1}f(v)-f(u)$, we have
  $$\Delta a(v)=\mathbf{1}[v=0].$$
  Hence, using Lemma \ref{escape}
  and the fact that $\pi\equiv 1$ is a stationary vector for the random walk,
  the function $v\mapsto h(v)-e_{b,c}[a(v-c)-a(v-b)]$ is bounded and harmonic,
  therefore constant, establishing \eqref{h-formula}.

We now claim that
for all $u,v$ with $\|u-v\|_1=1$,
\begin{equation}\label{h-diff}
|h(u)-h(v)|\leq C_3|v|^{-2}.
\end{equation}
Once this is established we obtain
\[\sum_{\substack{u,v\in B(k): \\ \|u-v\|_1=1}} |h(u)-h(v)| \leq
\sum_{j=1}^k C_4 j(C_3 j^{-2}) \leq C\ln k.\]
as required.

Finally, turning to the proof of \eqref{h-diff}, combining
\eqref{h-formula} and \eqref{asym} gives
$$h(u)-h(v)=C_5\big(\ln|u-c|-\ln|u-b|-\ln|v-c|+\ln|v-b|\big)+O(|v|^{-2}).$$
In order to bound the above expression, fix $u-v$ to be one of the 4
possible integer unit vectors, and write $v-c=z$ and $c-b=\alpha$ and
$u-v=\beta$.  For convenience identify the vector $(x,y)$ with the
complex number $x+iy$ and let $|\cdot|$ denote the modulus.  We have
\begin{align*}
&\ln|u-c|-\ln|u-b|-\ln|v-c|+\ln|v-b| \\
=&\ln\bigg|\frac{(z+\alpha)(z+\beta)}{(z+\alpha+\beta)z}\bigg|
=\ln\Big|1+\frac{\alpha}{z}+\frac{\beta}{z}
-\frac{\alpha}{z}-\frac{\beta}{z}+O(|z|^{-2})\Big| \\
=&O(|z|^{-2}) \qquad\text{ as }z\to\infty. \qedhere
\end{align*}
\end{proof}

 Fix $a,b,c\in \Z^2$, and consider the rotor walk $x_0,x_1,\dots$ started at
 $a$ with rotor mechanism \eqref{z2-mech} and rotor configuration \eqref{z2-r}
 modified so that $p(b,a)=p(c,a)=1$ as discussed in the paragraph preceding
 the statement of Theorem \ref{hit-prob-log}.
 We say that the walk {\dof enters a new layer} at time $t$
 if for some $k$ we have $x_0,\ldots,x_{t-1}\in B(k)$ but $x_t\not\in B(k)$.

\begin{lemma}\label{layers}
Under the above assumptions,
between any two times at which the rotor walk enters a new layer,
it must visit vertex $a$ at least once.
Also, between any two consecutive visits to vertex $a$,
no vertex is visited more than $4$ times.
\end{lemma}
\begin{proof}
  We start by proving the first assertion.
  The reader may find it helpful to consult Figure \ref{swastika} throughout.
  Suppose for a contradiction that \nopagebreak $a\in B(k-1)$,
  and that the rotor walk enters
  both the layers $\partial B(k)$ and $\partial B(k+1)$ for the first time
  without visiting $a$ in between.
  Let $\tmin$ be the time of the last visit to $a$
  prior to entering $\partial B(k)$,
  and let $\tmax$ be the first time at which $\partial B(k+1)$ is entered.

  We claim that some vertex $v$ emitted the particle
  at least $5$ times during $[\tmin,\tmax]$.
  To prove this, note first that $x_{\tmax-1}\in \partial B(k)$,
  and consider the following two cases.
  If $x_{\tmax-1}$ is not one of the four
  ``corner vertices'' of $\partial B(k)$,
  then immediately after the particle moves
  from $x_{\tmax-1}$ to $x_\tmax\in \partial B(k+1)$,
  the rotor at $x_{\tmax-1}$ is pointing in the same direction
  as in the initial rotor configuration $r$.
  Since this rotor did not move before time $\tmin$,
  vertex $x_{\tmax-1}$ must have emitted the particle
  at least $4$ times during $[\tmin,\tmax]$.
  Therefore, $x_{\tmax-1}$ must have received the particle at least $4$ times
  from among its $4$ neighbors in $[\tmin,\tmax]$ ---
  but it has not received the particle from $x_{\tmax}$,
  therefore by the pigeonhole principle
  it received it at least twice from some other neighbor $v$.
  And $v\not \in\{b,c\}$ since $x_{\tmax-1}\neq a$.
  By considering the rotor at $v$,
  we see that this implies that $v$ emitted the particle at least $5$ times
  during $[\tmin,\tmax]$.
  On the other hand, if $x_{\tmax-1}$ is a corner vertex of $\partial B(k)$,
  then on comparing with the initial rotor configuration $r$
  we see that $x_{\tmax-1}$ has emitted (and hence received) the particle
  $3$ or $4$ times, but two of its neighbors lie in $\partial B(k+1)$,
  so it did not receive the particle from them,
  and the same argument now applies.  Thus we have proved the above claim.

  Now let $u$ be the first vertex
  to emit the particle $5$ times during $[\tmin,\tmax]$.
  Then $u\not\in\{a,b,c\}$,
  otherwise we would have a contradiction to our assumption
  that $a$ is visited only once.
  But now repeating the argument above,
  $u$ must have received the particle $5$ times,
  so it must have received it at least twice from some neighbor,
  not in $\{a,b,c\}$,
  so this neighbor must have emitted the particle $5$ times
  by some earlier time in $[\tmin,\tmax]$, a contradiction.
  Thus the first assertion is established.

  The second assertion follows by an almost identical argument:
  if some vertex is visited at least $5$ times between visits to $a$,
  then considering the first vertex to be so visited leads to a contradiction.
\end{proof}

\begin{proof}[Proof of Theorem \ref{hit-prob-log}]
We write $C_1,C_2,\ldots$ for constants which may depend on $a,b,c$.
We use the proof of Proposition \ref{key} in the case $f=h$.
As in the proof of Theorem \ref{hit-prob},
equation \eqref{almost-done} becomes
$$h(a)n-n_t(b)=h(x_t)-h(a)+
\sum_{u\in V\setminus\{b,c\}} \big[\phi(u,r_t(u))-\phi(u,r_0(u))\big],$$
where $n=n_t:=n_t(b)+n_t(c)$.
However, the term $\phi(u,r_t(u))-\phi(u,r_0(u))$ is non-zero
only for those vertices which have been visited by time $t$.
Now the first assertion of Lemma \ref{layers} implies that
at most one new layer is entered for each visit to $a$,
and thus for each visit to $\{b,c\}$.
Hence for some $C_1$, all the vertices visited by time $t$ lie in $B(n+C_1)$
(where the constant $C_1$ depends on the layer of the initial vertex $a$).

Now proceeding as in the proof of Proposition \ref{key}
and using Lemma \ref{log-bound},
\begin{align*}
  \Big|\sum_{u\in B(n+C_1)\setminus\{b,c\}}
\big[\phi(u,r_t(u))-\phi(u,r_0(u))\big]\Big|& \leq
    \tfrac12\sum_{\substack{u,v\in B(n+C_1+1): \\ \|u-v\|_1=1}}
    |h(u)-h(v)| \\
  &\leq C\ln n.
\end{align*}
Combining this with the above facts gives
\[\big|h(a)n-n_t(b)\big|\leq 1+C\ln n,\]
as required.

Finally to prove the bound $t\leq C'n^3$,
we note by the second assertion of Lemma \ref{layers}
that after $n$ visits to vertex $a$,
each of the at most $C_2 n^2$ vertices in $B(n+C_1)$
has been visited at most $4n$ times,
so the total number of time steps is at most $4 C_2 n^3$.
\end{proof}

\section{Proofs for transfinite walks} \label{proofs-for-transfinite-walks}

\begin{proof}[Proof of Lemma \ref{rec-trans}]
By irreducibility it is enough to show that
if $u$ is visited infinitely often and $p(u,v)>0$
then $v$ is visited infinitely often.
But this is immediate since $v=u^{(i)}$ for some $i$,
so the rotor at $u$ will be incremented to point to $v$ infinitely often.
\end{proof}

\begin{proof}[Proof of Lemma \ref{trans-rec-trans}]
As in the preceding proof, if $u$ is visited infinitely often and $p(u,v)>0$
then $v$ is visited infinitely often, proving the first assertion.
For the second assertion, let $M$ be one greater than the first $m$
for which the walk $x_{m\omega},x_{m\omega+1},\ldots$ is recurrent,
or $M=\omega$ if all are transient.
Then $a$ is visited infinitely often before time $M\omega$,
and we apply the first assertion.
\end{proof}

\begin{proof}[Proof of Theorem \ref{hit-prob-tf}]
We consider the quantity $\Phi$ defined in the proof of Proposition \ref{key},
with $f=h=h_{b,c}$ (as in the proof of Theorem \ref{hit-prob}).
Suppose $x_0,x_1,\ldots$ is a transient rotor walk.  We claim that
\begin{equation}\label{Phi-omega}
\Phi(x_\omega,r_\omega)-\lim_{t\to\infty}\Phi(x_t,r_t)=h(a).
\end{equation}
The claim is proved as follows. The assumption of the theorem and
the fact that the walk is transient imply that $\lim_{t\to\infty}
h(x_t)=0$.  We clearly have $\lim_{t\to\infty}
\phi(u,r_t(u))=\phi(u,r_\omega(u))$ for each $u$, and by \eqref{term-bound}
and the definition of $\Khitprob$ in Theorem \ref{hit-prob} we
have for all $u$ and $t$ that $|\phi(u,r_t(u))-\phi(u,r_0(u))|\leq F(u)$
where $\sum_{u \in V} F(u) \leq 2(\Khitprob-1)$.
Hence by the dominated convergence theorem,
$$\lim_{t\to\infty}\Phi(x_t,r_t)=0+\sum_{u\in V}
\big[\phi(u,r_\omega(u))-\phi(u,r_0(u))\big]=\Phi(x_\omega,r_\omega)-h(a).$$

We have proved claim \eqref{Phi-omega};
thus whenever we ``restart from infinity to $a$'',
the quantity $\Phi$ increases by $h(a)$.
Combining this with the argument
from the proof of Theorem~\ref{hit-prob}, we get
$$\big[n_\tau(b)+n_\tau(c)+m\big]h(a)-n_\tau(b)
=\Phi(x_\tau,r_\tau)-\Phi(x_0,r_0)$$
for $\tau=m\omega+t$,
and the right side is bounded in absolute value by $\Khitprob$
exactly as in the proof of Theorem~\ref{hit-prob}.
\end{proof}

\begin{proof}[Proof of Theorem \ref{num-vis-tf}]
We consider the quantity $\Phi$ defined in the proof of Proposition \ref{key},
with $f=g=\numvis_b$.  Note that
  $$\Delta g(u)=\begin{cases}
  0,&u\neq b;\\
  -1,&u=b.
  \end{cases}$$
Mimicking the proof of Theorem~\ref{hit-prob-tf},
we obtain
$$g(a)m-n_\tau(b)=\Phi(x_\tau,r_\tau)-\Phi(x_0,r_0),$$
and we bound the right side as in the previous proofs,
noting that when $u=b$ we have $|g(u)-g(v)+\Delta g(u)|\leq |g(u)-g(v)|+1$.
\end{proof}

Our proof of Theorem \ref{density} is based on an unpublished argument of
Oded Schramm (although we present the details in a somewhat different way).
We will need some preparation. It will be convenient to work with
$R_n:=n-I_n$, i.e.\ the number of times the transfinite rotor walk returns to
$a$ {\em without} going to infinity up to the time of the $n$th return to
$a$. We also introduce some modified Markov chains and rotor mechanisms as
follows.

Firstly, replace the vertex $a$ with two vertices $a_0$ and $a_1$. Let
$\widehat{V}=(V\setminus\{a\})\cup\{a_0,a_1\}$ denote this modified vertex
set.  Introduce a modified transition kernel $\widehat{p}$ by letting
$a_0$ inherit all the outgoing transition probabilities from $a$, and
letting $a_1$ inherit all the incoming transition probabilities to $a$
(i.e.\ let $\widehat{p}(a_0,v)=p(a,v)$ and $\widehat{p}(v,a_1)=p(v,a)$ for all
$v\in V\setminus \{a\}$); also let $\widehat{p}(a_1,a_0)=1$ and
$\widehat{p}(a_0,a_1)=0$, and let $\widehat{p}$ otherwise agree with $p$.

Secondly, for a positive integer $d$, let $B(d)$ denote the set of
vertices that can be reached in at most $d$ steps of the original
Markov chain starting from $a$, and let $\partial B(d):=B(d)\setminus
B(d-1)$.  Let $\widehat{p}^d$ be $\widehat{p}$ modified so that
$\widehat{p}^d(b,a_0)=1$ for all $b\in \partial B(d)$.  (Thus, on reaching
distance $d$ from $a$, the particle is immediately returned to $a_0$).

Fix a rotor mechanism and initial rotor configuration for the original
Markov chain, and modify them accordingly to obtain a rotor walk
associated with $\widehat{p}^d$, started at $a_0$.
Let $R_n^{d}$ be the number of times
this rotor walk hits $a_1$ before the $n$th return to $a_0$ (i.e.\
before the $(n+1)$st visit to $a_0$).  Also note that $R_n$ is the
number of times the transfinite rotor walk associated with $\widehat{p}$
and started at $a_0$ hits $a_1$ before the $n$th return to $a_0$.

\begin{lemma}
\label{conv}
For a fixed initial rotor configuration,
and any non-negative integer $n$,
we have $R_n^{d}\to R_n$ as $d\to\infty$
(i.e., $R_n^{d}= R_n$ for $d$ sufficiently large).
\end{lemma}

\begin{proof}
For $v\in\widehat{V}$, let $N^d_n(v)$ (respectively $N_n(v)$) be the
number of visits to vertex $v$ before the $n$th return to $a_0$ for
the (transfinite) rotor walk associated with $\widehat{p}^d$ (respectively
$\widehat{p}$).  We claim that
\begin{equation}\label{prod-conv}
N^d_n \to N_n \qquad\text{as }d\to\infty,
\end{equation}
where the convergence is in the product topology on $\N^{\widehat{V}}$; in
other words, for any finite set $F\subset \widehat{V}$, if $d$ is
sufficiently large then $N^d_n(v)=N_n(v)$ for all $v\in F$.  The
required result follows immediately from this, because
$R_n^d=N_n^d(a_1)$ and $R_n=N_n(a_1)$.

We prove \eqref{prod-conv} by induction on $n$.  It holds trivially
for $n=0$ because $N^d_0$ and $N_0$ equal zero everywhere.  Assume it
holds for $n-1$.  This implies in particular that the rotor
configuration at the time of the $(n-1)$st return to $a_0$ similarly
converges as $d\to\infty$ to the corresponding rotor configuration
in the transfinite case.  Now consider the portion of the transfinite
rotor walk corresponding to $\widehat{p}$, starting just after the $(n-1)$st
return to $a_0$, up until the $n$th return to $a_0$.  Consider the
following two possibilities.  If this walk is recurrent (so that it
returns to $a_0$ via $a_1$) then it visits only finitely many
vertices, so if $d$ is sufficiently large that $N^d_{n-1}$ and $N_{n-1}$
agree on all the vertices it visits, then $N^d_{n}$ and $N_{n}$ agree
also agree on the same set of vertices, establishing \eqref{prod-conv}
in this case.  On the other hand, suppose the aforementioned walk is
transient (so that it goes to infinity before being restarted at
$a_0$). Given a finite set $F\subset \widehat{V}$, let $d$ be such that
that when this walk leaves $F$ for the last time, it has never been
outside $B(d)$.  Now let $d'$ be such that $N^{d'}_{n-1}$ and
$N_{n-1}$ agree on $B(d)$.  Then $N^d_{n}$ and $N_{n}$ will agree on
$F$.  So \eqref{prod-conv} holds in this case also,
and the induction is complete.
\end{proof}

\begin{lemma}
\label{mono}
For all positive integers $n$ and $d$ we have $R_n^{d+1}\geq R_n^{d}$.
\end{lemma}

\begin{proof}
This will follow by a special case of the Abelian property for rotor
walks on finite graphs with a sink (see e.g.\ \cite[Lemma
3.9]{hlmppw}).  First we slightly modify the mechanism yet again.
Consider the rotor mechanism and initial rotor configuration
corresponding to $\widehat{p}^{d+1}$.  Remove all the vertices in
$V\setminus B(d+1)$ (these cannot be visited by the rotor walk started
at $a_0$ anyway).  Introduce an additional {\em absorbing} vertex $s$
(called the sink), and modify the transition
probabilities so that on hitting $a_1$ or $\partial B(d+1)$, particles
are sent immediately to $s$ instead of to $a_0$.  Modify the rotor
mechanism accordingly, but do not otherwise modify the initial rotor
configuration.

We now consider the following multi-particle rotor walk (see e.g.\
\cite{hlmppw} or the discussion in the introduction for more
information).  Start with $n$ particles at $a_0$, and perform a
sequence of rotor steps.  That is, at each step, choose any non-sink
vertex which has a positive number of particles (if such exists), and
{\dof fire} the vertex; i.e.\ increment its rotor, and move one
particle in the new rotor direction.  Continue in this way until all
particles are at the sink. \cite[Lemma 3.9]{hlmppw} states that the
total number of times any given vertex fires during this procedure is
independent of our choices of which vertex to fire.

In particular, consider the firing order in which we first move one
particle repeatedly (so it performs an ordinary rotor walk) until it
reaches $s$, then move the second particle in the same way, and so on.
Thus the number of times $a_1$ fires is $R_n^{d+1}$.  Alternatively,
we may move one particle until the first time it reaches $\partial
B(d) \cup \{s\}$, then ``freeze'' it, and move the second particle
until it reaches $\partial B(d)\cup \{s\}$, and so on.  At this stage,
the number of times $a_1$ has fired is $R_n^{d}$.  Now we can continue
firing until the frozen particles reach $s$.  Comparing the two
procedures shows $R_n^{d+1}\geq R_n^{d}$.
\end{proof}

\begin{cor}
\label{abel}
For all positive integers $n$ and $d$ we have $R_n\geq R_n^{d}$.
\end{cor}

\begin{proof}
Immediate from Lemmas \ref{conv} and \ref{mono}.
\end{proof}

\begin{proof}[Proof of Theorem \ref{density}]
Since $R_n=n-I_n$, the required result is clearly equivalent to
$\liminf_{n\to\infty}R_n/n\geq \P_a(T^+_a<\infty)$.  Fix any
$\epsilon>0$.  Then there exists $d$ such that $\P_a(T^+_a<T_{\partial
B(d)})\geq \P_a(T^+_a<\infty)-\epsilon$.  Now consider the modified
rotor walk corresponding to $\widehat{p}^d$ as defined above.  Since the
set of vertices that can be reached from $a_0$ is finite (so in effect
the vertex set is finite), Theorem \ref{hit-prob} implies that
$R_n^d/n \to \P_a(T^+_a<T_{\partial B(d)})$ as $n\to\infty$.  Putting
these facts together with Corollary \ref{abel} we obtain
\[\liminf_{n\to\infty} \frac{R_n}{n}\geq \lim_{n\to\infty} \frac{R_n^d}{n}
=\P_a(T^+_a<T_{\partial B(d)}) \geq \P_a(T^+_a<\infty)-\epsilon.\qedhere\]
\end{proof}

\section{Proofs for stack walks} \label{proofs-for-stack-walks}

In this section we will prove Proposition \ref{seq},
and use it together with Proposition \ref{stack-key} below
to prove Theorem \ref{stack}.
Given a Markov chain and a stack mechanism, we define the discrepancy functions
$$D_n(u,v):=\#\{i\leq n: u^{(i)}=v\}-np(u,v).$$
\begin{prop}\label{stack-key}
For any Markov chain, any stack walk, any function $f$ and any $t$,
$$\sum_{s=0}^{t-1} \Delta f(x_s)
=f(x_t)-f(x_0)+\sum_{u,v\in V} D_{n_t(u)}(u,v)\big[f(u)-f(v)+\Delta f(u)\big].$$
\end{prop}

\begin{proof}
Consider the function
$$\Psi(t):=f(x_t)+\sum_{u\in V} \psi(u,n_t(u))$$
where
$$\psi(u,n):=\sum_{i=1}^n\big[f(u)-f(u^{(i)})+\Delta f(u)].$$
As in the proof of Proposition \ref{key} we have
$\sum_{s=0}^{t-1}\Delta f(x_s)=\Psi(t)-\Psi(0)$.

From the definition of $D$ we have
$$
\psi(u,n)=
\sum_{v\in V}\big[D_n(u,v)+np(u,v)\big]\big[f(u)-f(v)+\Delta f(u)\big].
$$
But by the definition of the Laplacian,
$\sum_{v\in V} p(u,v)[f(u)-f(v)+\Delta f(u)]=0$; therefore
$$
\psi(u,n)=\sum_{v\in V}D_n(u,v)\big[f(u)-f(v)+\Delta f(u)\big],
$$
and the result follows on substituting.
\end{proof}

\begin{proof}[Proof of Proposition \ref{seq}]
First note that it suffices to prove the case
in which $p_1,\ldots,p_n$ are all rational.
The irrational case then follows by a limiting argument.
Specifically, let $p_1^k,\ldots,p_n^k$ be rational
with $p_i^k\to p_i$ as $k\to\infty$,
and let $z^k_1,z^k_2,\ldots$ be a sequence satisfying \eqref{ineq}
for the $p_i^k$'s;
then by a compactness argument (since $\{1,\ldots,n\}$ is finite)
there is a subsequence $(k_j)$ and a sequence $z_1,z_2\ldots$ such that
$z^{k_j}_t\to z_t$ for each $t$,
and $(z_t)$ then satisfies \eqref{ineq} for the $p_i$'s.

Now suppose that $p_1,\ldots,p_n$ are rational,
and let $d$ be their least common denominator.
Consider the finite bipartite graph $G$
with vertex classes $L:=\{1,\ldots,d\}$ and
$R:=\bigcup_{i=1}^n R_i$ where $R_i:=\{(i,m): m\in\{1,\ldots,p_id\}\}$,
and with an edge from $t\in L$ to $(i,m)\in R$ if and only if
\begin{equation}\label{edge}
\Big\lceil \frac{m-1}{p_i}\Big\rceil \leq
t \leq \Big\lceil \frac{m}{p_i}\Big\rceil.
\end{equation}

We will prove that $G$ has a perfect matching between $L$ and $R$.
Note first that $\#L=d=\sum_i p_i d=\#R$.
We claim that any set $T\subseteq L$ has at least $p_i\#T$ neighbors in $R_i$;
from this it follows that it has at least $\#T$ neighbors in $R$,
and the existence of a perfect matching
then follows from Hall's marriage theorem.
To prove the claim, fix $i\in\{1,\ldots,n\}$ and note that
\eqref{edge} is equivalent to $p_it-p_i<m\leq p_it+1$.
Therefore in the case when $T$ is an interval $[s,t]$,
it is adjacent to all those pairs $(i,m)\in R_i$
for which $m$ is an integer in $(p_is-p_i,p_it+1]\cap [1,d]$;
this includes all integers in $(p_i s-p_i,p_it+1)$
(since $p_is-p_i\geq 0$ and $p_it+1\leq d+1$).
The latter interval has length $p_i(t-s+1)+1$,
therefore it contains at least $p_i(t-s+1)=p_i\#T$ integers as required.
Now consider the case $T=[s,t]\cup[u,v]$ (where $u>t+1$).
If the two intervals have disjoint neighborhoods in $R_i$,
the claim follows by applying the single-interval case to each and summing.
On the other hand if the neighborhoods of the two intervals intersect,
we see from \eqref{edge} that the neighborhood in $R_i$ of $T$
is the same as the neighborhood in $R_i$ of the larger set $[s,v]$,
so the claim again follows from the single-interval case.
Finally, the case when $T$ is a union of three or more intervals
is handled by applying the same reasoning to each adjacent pair,
proving the claim and hence the existence of a matching.

Fix a perfect matching of $G$, and for $t=1,\ldots,d$,
let $z_t:=i$ where $R_i$ contains the partner of $t$.
It follows from \eqref{edge} that
if $(i,m)$ and $(i,m+1)$ have respective partners $t$ and $t'$ then $t<t'$.
Therefore $t$ and $(i,m)$ are partners if and only if
$z_t$ is the $m$th occurrence of $i$ in the sequence $z$;
from  \eqref{edge} this $m$th occurrence appears between
positions $\lceil \frac{m-1}{p_i}\rceil$ and $\lceil \frac{m}{p_i}\rceil$.
Thus,
\[ \lfloor p_it\rfloor \leq \#\{s\leq t: z_s=i\} \leq \lfloor p_it\rfloor+1,\]
and it follows that \eqref{ineq} holds for all $t\leq d$.
Note also that the left side of \eqref{ineq} is zero for $t=d$,
therefore continuing the sequence $z$
so as to be periodic with period $d$ completes the proof.
\end{proof}

(For an alternative proof of Proposition~\ref{seq} that applies
also to infinite probability vectors, see~\cite{ahmp}.)

\begin{proof}[Proof of Theorem \ref{stack}]
Choosing the stack mechanism according to Proposition \ref{seq}
ensures that $|D_n(u,v)|\leq 1$ for all $u$, $v$ and $n$.
Now apply Proposition~\ref{stack-key} to $f=h$ to obtain
\begin{multline*}
\big|h(a)[n_t(b)+n_t(c)]-n_t(b)\big|
\\
\leq |h(x_t)-h(x_0)|+\sum_{\substack{u\in V\setminus\{b,c\}, \\ v\in V }}
|D_{n_t(u)}(u,v)|\cdot |h(u)-h(v)|,
\end{multline*}
and conclude by noting that $D_{n_t(u)}(u,v)=0$ unless $p(u,v)>0$.
\end{proof}

\section*{Open questions} \label{open-questions}

As the burgeoning literature on Eulerian walk and rotor-routing
attests, there are numerous interesting open problems. Here we
focus on a few that are related to rotor walk on Euclidean
lattices.

\begin{mylist}
\item Can the bound $C\log n/n$ in Theorem \ref{hit-prob-log}
    for the discrepancy in hitting probabilities for simple
    random walk on $\Z^2$ be improved to $C/n$?  Do similar
    results hold in $\Z^d$ and for other initial rotor
    configurations?
\item For simple random walk on $\Z^2$, let all rotors
    initially point East, and consider the
    transfinite rotor walk restarted at the origin after each
    escape to infinity.
    What is the asymptotic behavior of $I_n$, the
    number of escapes to infinity before the $n$th return to the
    origin?  Theorem \ref{density} implies that $I_n/n\to 0$
    as $n\to\infty$, but simulations suggest that the
    convergence is rather slow.
\item For simple random walk on $\Z^d$ with $d\geq 3$, does
    there exist an initial rotor configuration for which the rotor
    walk is recurrent?
\end{mylist}

\section*{Acknowledgments}

We thank Omer Angel, David desJardins, Lionel Levine, Russell Lyons,
Karola M\'esz\'aros, Yuval Peres, Oded Schramm and David Wilson
for valuable discussions.

\bibliography{rr-rev}

\noindent
{\sc Alexander E. Holroyd:}
{\tt holroyd at math dot ubc dot ca}\\
Microsoft Research, 1 Microsoft Way, Redmond WA, USA; and
\\ University of British Columbia, 121-1984 Mathematics Rd.,
Vancouver, BC, Canada.

\vspace{3mm}
\noindent
{\sc James Propp:}
{\tt jpropp at cs dot uml dot edu}\\
University of Massachusetts Lowell,
1 University Ave., Olney Hall 428, Lowell, MA,  USA.

\end{document}